\documentclass[10pt,reqno,a4paper]{amsart}
\usepackage[USenglish]{babel}
\usepackage[a4paper,margin=1in]{geometry}

\usepackage{amsmath,amsthm,amssymb}
\usepackage[dvipsnames]{xcolor}
\usepackage{xcolor}
\usepackage{color}

\usepackage{mathscinet}
\usepackage{verbatim}
\usepackage{enumerate}
\usepackage{mathrsfs}
\usepackage[colorlinks,citecolor=red]{hyperref}
\usepackage[numbers,sort]{natbib}

\newtheorem{definition}{Definition}[section]
\newtheorem{theorem}{Theorem}[section]
\newtheorem{lemma}{Lemma}[section]
\newtheorem{remark}{Remark}[section]

\newtheorem{proposition}{Proposition}[section]

\newcommand{\RR}{\mathbb{R}}

\newcommand{\seq}[1]{\left<#1\right>}

\providecommand{\abs}[1]{\lvert#1\rvert}
\providecommand{\norm}[1]{\lVert#1\rVert}

\title{Blow-up for a non-linear stable non-Gaussian process in fractional time}
\date{}

\author[]{Soveny Sol{\'i}s and  Vicente Vergara}


\address{S. Sol\'is: Departamento de Matem\'aticas, Facultad de Ciencias F\'isicas y Matem\'aticas\\
Universidad de Concepci\'on,  Concepci\'on, Chile.}
\email{ssolis@udec.cl}

\address{S. Sol\'is: Departamento de Matem\'aticas, Facultad de Ciencias Naturales y Matem\'aticas\\
Escuela Superior Polit\'ecnica del Litoral, Guayaquil, Ecuador.}
\email{ssolis@espol.edu.ec}

\address{V. Vergara: Departamento de Matem\'aticas, Facultad de Ciencias F\'isicas y Matem\'aticas\\
	Universidad de Concepci\'on,  Concepci\'on, Chile.}
\email{vvergaraa@udec.cl}

\thanks{The authors were partially supported by Chilean research grant Fondo Nacional de Desarrollo Científico y Tecnológico, FONDECYT 1190255.}

\thanks{Corresponding author: S. Sol\'is}

\begin{document}

\maketitle

\begin{abstract}
The behaviour of solutions for a non-linear diffusion problem is studied. A subordination principle is applied to obtain the variation of parameters formula in the sense of Volterra equations, which leads to the integral representation of a solution in terms of the fundamental solutions. This representation, the so-called mild solution, is used to investigate some properties about continuity and non-negativeness of solutions as well as to prove a Fujita type blow-up result. Fujita's critical exponent is established in terms of the parameters of the stable non-Gaussian process and a result for global solutions is given.
\end{abstract}

\section{Introduction}
\label{intro}

Let $\gamma>1$. We consider the following Cauchy problem 
\begin{equation}
\label{general}
\begin{split}
\partial^{\alpha}_t(u-u_0)(t,x)+\Psi_{\beta}(-i\nabla)u(t,x) & =\abs{u(t,x)}^{\gamma-1}u(t,x),\quad t>0, \; x\in\mathbb{R}^d,\\ 
u(t,x)|_{t=0} & = u_0(x)\geq 0,\quad x\in\mathbb{R}^d, 
\end{split}
\end{equation}
where $\partial^{\alpha}_t$ is the \textit{Riemann–Liouville fractional derivative} of order $\alpha\in(0,1)$ given by
\[
\partial^{\alpha}_t v=\dfrac{d}{dt}\displaystyle\int_0^t g_{1-\alpha}(t-s)v(s)ds=:\dfrac{d}{dt}(g_{1-\alpha}*v)(t),
\]
with $g_{\rho}(t):=\frac{1}{\Gamma(\rho)}t^{\rho-1}$ and the Euler Gamma function $\Gamma(\cdot)$. The notation $\partial^{\alpha}_t(u-u_0)(t,x)$ is understood as $\partial^{\alpha}_t u(t,x)-\partial^{\alpha}_t (1)\times u_0(x)$ and the function $u_0$ stands for the initial data in a certain Lebesgue space. The term $\Psi_{\beta}(-i\nabla)$ is a \textit{singular integral operator} of constant order $\beta\in(0,2)$ with symbol $\psi$, that is
\[
\Psi_{\beta}(-i\nabla) v(x)= \mathcal{F}^{-1}_{\xi\to x}[\psi(\xi) (\mathcal{F}v)(\xi)], \, v\in C_0^{\infty}(\RR^d),
\]
where $\mathcal{F}$ is the Fourier transform in $\RR^d$ and $\mathcal{F}^{-1}$ is its inverse. As usual, $C_0^{\infty}(\RR^d)$ denotes the space of test functions on $\RR^d$. The symbol $\psi$ is a measurable function on $\RR^d$ given by
\[
\psi(\xi)=\norm{\xi}^\beta \omega_\mu\left(\frac{\xi}{\norm{\xi}}\right),\quad \xi\in\mathbb{R}^d,
\]
where
\begin{equation*}
\omega_\mu(\theta):=\displaystyle\int_{S^{d-1}}\abs{\theta\cdot\eta}^\beta\mu(d\eta),\quad\theta\in S^{d-1},
\end{equation*}
with $\norm{\xi}=\sqrt{\xi_1^2+\cdots+\xi_d^2}$ being the standard Euclidean norm. Here, $\mu(d\eta)$ is a centrally symmetric finite (non-negative) Borel measure defined on the unitary sphere $S^{d-1}$, called \textit{spectral measure}, and $\omega_\mu(\cdot)$ is a continuous function on $S^{d-1}$, see e.g. \cite[Section 1.8]{Kol19}. Whenever $\mu(d\eta)=\varrho(\eta)d\eta$, where $\varrho$ is a continuous function on $S^{d-1}$, we will refer to $\varrho$ as the \textit{density of $\mu$}. Some restrictions on the function $\varrho$ may be required for the lower bound and behaviour of the fundamental solutions; see, e.g. \cite[Section 5.2]{KolV00}. More precisely, our basic hypothesis throughout the paper is the following:
\begin{itemize}
	\item[$(\mathcal{H}_1)$] The spectral measure $\mu$ has a strictly positive density, such that the function $\omega_\mu$ is strictly positive and $(d+1+[\beta])$-times continuously differentiable on $S^{d-1}$.
\end{itemize}
We denote by $(\mathcal{H}_2)$ to refer to $(\mathcal{H}_1)$ whenever we need to assume that $\omega_\mu$ is $(d+2+[\beta])$-times continuously differentiable on $S^{d-1}$, $[\beta]$ being the maximal integer not exceeding the real number $\beta$. The considerations just made above have been taken from \cite[Proposition 4.5.1]{Kol19} and \cite[Theorem 4.5.1]{Kol19}, for $d=1$ and $d>1$ respectively. We want to point out that the condition of strict positivity on $\omega_\mu$ in $\mathcal{H}_1$, guarantees that the support of the measure $\mu$ on $S^{d-1}$ is not contained in any hyperplane of $\RR^d$ (\cite[Section 4.5]{Kol19}).

In this work we are concerned with studying the blow-up phenomena for reaction-diffusion equations like \eqref{general} considering a non-regular class of solutions instead of the classical calls, with a temporal fractional derivative and a pseudo-differential operator related to a stochastic process.

Our time derivative is also called the \textit{Caputo fractional derivative}, for instance, in the sense of \cite[Section 2.4]{KST06} with $0<\alpha<1$. However, other authors may require smoothness conditions on the function to define the Caputo derivative (see, e.g. \cite[Sub-section 2.4.1]{Pod99}). Since $\psi$ is a Lévy-Khintchine symbol with index of stability $\beta$, it is well known that the corresponding stochastic process is called a localised Feller-Courrège process and therefore it makes sense to use the notation $\Psi_{\beta}(-i\nabla)$ for the associated generator (see, e.g. \cite[Chapter 6, Appendices C and D]{KolV00}).

As an important concept associated with $\alpha$, let us mention the \textit{mean squared displacement} ($MSD$) or the \textit{centred second moment}, which describes how fast is the dispersion of the particles in a random process. In \cite[Lemma 2.1]{KSVZ16}, the authors proved that the $MSD$ governed by the equation
\[
\partial^{\alpha}_t(u-u_0)-\Delta u=0
\]
specifically turns out to be $\frac{2d}{\Gamma(1+\alpha)}t^{\alpha}$, $t>0$, $0<\alpha<1$. In the literature one traditionally finds that anomalous diffusion refers to this power-law. See, e.g. \cite{SW89}, \cite{MMPG07}, \cite{Luc14}, \cite{AM22} and references therein. However, in  our case, the Cauchy problem \eqref{general} do not possess a finite $MSD$. This can be directly checked by using the definition of $MSD$ (\cite[expression (6)]{KSVZ16}) and similar arguments as in the proof of \cite[Lemma 2.1]{KSVZ16} or \cite[Theorem 2.8]{SV22}.  

In this setting, the Green function of the Cauchy problem $\dfrac{\partial u}{\partial t}+\Psi_{\beta}(-i\nabla) u=0$ is non-Gaussian and it is interpreted as the transition probability density of the corresponding \textit{stable non-Gaussian process} \cite[Chapter 7]{Kol11}. The study of these processes and their generalizations is motivated by the increasing use in the mathematical modeling of processes in engineering, natural sciences and economics. See, e.g. \cite{ASMC21}, \cite{AM22}, \cite{MK00} and \cite[Chapter 1]{Zol86}. Similar to the case of Gaussian processes, which have been widely studied (see, e.g. \cite{Aro67}, \cite{Fri83}, \cite{HKNS06}, \cite{Bei11}, \cite{IKO02}), it arises an interest in qualitative properties, blow-up and asymptotic behaviour for the solutions of non-Gaussian ones. For instance, in the case $\beta=2$ and $\omega_\mu\equiv 1$ we see that the operator, namely $\Psi_2(-i\nabla)$, becomes the negative Laplacian $(-\Delta)$ with symbol $\psi(\xi)=\norm{\xi}^2$. The blow-up of the solution to the corresponding ordinary differential equation
\begin{equation*}
\begin{split}
\partial_t\;u(t,x)+(-\Delta)u(t,x) & = u(t,x)^{\gamma},\quad t>0, \; x\in\mathbb{R}^d,\\ 
u(t,x)|_{t=0} & = u_0(x)\geq 0,\quad x\in\mathbb{R}^d, 
\end{split}
\end{equation*}
was investigated by Fujita in 1966 (\cite{Fuj66}). Since then, many other researchers have explored blow-up phenomena (see, e.g. \cite{KLT05}, \cite{LW18}, \cite{NZZG18}, \cite{QP07}, \cite{VZ17}, \cite{ZS15}). An interesting generalization that includes a Riemann-Liouville fractional integral in the non-linear term, on the right hand side, can be found in \cite{LZ18}. Following the analysis of this phenomenon, in this work we show that the non-linearity of \eqref{general} leads to the blow-up of positive solutions in a finite time. 

For this purpose, we say that a function $u:[0,T)\times\RR^d\rightarrow\RR$ \textit{blows-up} at the finite time $T$ if 
\[
\lim_{t\rightarrow T^-}\norm{u(t)}_\infty =+\infty,
\]
and thus our main result is stated as follows. 
\begin{theorem}
\label{blowUp}
Let $\alpha\in(0,1)$ and $\beta\in (0,2)$. Assume the hypothesis $(\mathcal{H}_1)$ holds. Suppose that $\alpha=\frac{\beta}{2}$, that $1<p<\infty$ and that $u_0\in L_p(\RR^d)\cap C(\RR^d)$ is a non-negative function. If $1<\gamma <1+ \frac{\beta}{d}$, then all non-trivial non-negative solutions of \eqref{general} that admit the representation \eqref{integral} can only be local. If $\gamma=1+ \frac{\beta}{d}$, then the non-trivial non-negative solutions can only be local whenever the initial condition is sufficiently large. Moreover, if additionally $u_0\in L_\infty(\RR^d)$, then any positive mild solution of \eqref{general} blows-up in finite time.
\end{theorem}
Since the literature on blow-up theorems of Fujita type is quite extensive, we do not attempt to review it in this paper. Nevertheless, let us emphasize that the relation $\alpha=\frac{\beta}{2}$ pays a crucial role in the proof of Theorem \ref{blowUp}, which makes a similarity with what Fujita (1928-) found in 1966 for the case $\alpha=1$ working in the Gaussian framework when $\beta=2$ and $\omega_\mu\equiv 1$. For our proof, we exploit the representation of the Volterra equations in the sense of Prüss as well as the theory of \textit{completely positive kernels} of type ($\mathcal{PC}$), like $g_\rho$ with $\rho\in (0,1)$ (see, e.g. \cite{PV18}). We also deal with the theory of pseudo-differential operators that generate a sub-Markovian semigroup on $L_p(\RR^d)$, under the condition that the symbol $\psi:\RR^d\rightarrow\mathbb{C}$ is a continuous and negative definite function; see \cite[Examples 4.1.12 and 4.1.13]{Jac01}. In our case, $\psi$ satisfies such condition (\cite[Formula 1.9]{Kol00}, \cite[Theorem 3.6.11 and Lemma 3.6.8]{Jac01}).

Other results in \cite[Section 2]{JK19} and \cite[Section 8.2]{Kol19}, are also particularly important for our work. Here, the authors show that the linear Cauchy problem
\begin{equation*}
	\begin{split}
		\partial^{\alpha}_t(u-u_0)(t,x)+\Psi_{\beta}(-i\nabla)u(t,x)&=f(t,x),\quad t>0, \; x\in\mathbb{R}^d,\\ u(t,x)|_{t=0}&=u_0(x),\quad x\in\mathbb{R}^d, 
	\end{split}
\end{equation*}
admits a pair of fundamental solutions $(Z,Y)$, given by
\begin{equation}\label{Z:G}
	Z(t,x):=\dfrac{1}{\alpha}\displaystyle\int_0^\infty G(t^\alpha s,x)s^{-1-\frac{1}{\alpha}}G_\alpha(1,s^{-\frac{1}{\alpha}})ds
\end{equation}
and
\begin{equation}\label{Y:G}
	Y(t,x):=\displaystyle\int_0^\infty t^{\alpha-1}G(t^\alpha s,x)s^{-\frac{1}{\alpha}}G_\alpha(1,s^{-\frac{1}{\alpha}})ds,
\end{equation}
where $G$ stands for the Green function that solves the problem
\begin{equation}\label{HomogeneaOrden1}
\begin{split}
\partial_t\, G(t,x)+\Psi_{\beta}(-i\nabla)G(t,x)&=0,\quad t>0, \; x\in\mathbb{R}^d,\\ 
G(t,x)|_{t=0}&=\delta_0(x),\quad x\in\mathbb{R}^d,
\end{split}
\end{equation}
$\delta_0$ being the Dirac delta distribution, and $G_\alpha(\cdot,\cdot)$ is the Green function that solves the problem 
\[
\partial_t\, v(t,s)+\dfrac{d^\alpha}{ds^\alpha}v(t,s)=0,\quad t>0,\;s\in\RR, \; G_\alpha(0,s)=\delta(s), 
\]
where $\alpha\in (0,1)$ and 
\[\dfrac{d^\alpha}{ds^\alpha}f(s):=\dfrac{1}{\Gamma(-\alpha)}\int_0^\infty \dfrac{f(s-\tau)-f(s)}{\tau^{1+\alpha}}d\tau,
\] 
see \cite[Formulas (1.111) and (2.74)]{Kol19}. From \cite[Lemma 2.15]{SV22}, we know that the fundamental solutions $(Z,Y)$ given by \eqref{Z:G}-\eqref{Y:G} satisfy the relation
\begin{equation}
\label{relacionYZ}
Y(\cdot,x)=\dfrac{d}{dt}(g_{\alpha}*Z(\cdot,x)),\quad t>0,\quad x\in\RR^d\setminus\{0\},
\end{equation}
which is crucial for the integral representation \eqref{integral} below.

This paper is organized as follows. In Section \ref{sec:1} we have compiled some properties of the fundamental solutions $Z,Y$. In Section \ref{sec:2} we derive the integral representation of a solution to \eqref{general} in the sense of Definition \ref{defSolution}. The main result of this section is given by Theorem \ref{localrepresentation}. In Section \ref{sec:3} we prove two results on continuity and non-negativeness of the local solutions, given by Theorems \ref{continuidad} and \ref{positividad} respectively. The last section is devoted to prove our main result related to the blow-up of positive solutions, stated in Theorem \ref{blowUp}. We also give a Fujita type result for global solutions in Theorem \ref{SolucionGlobal}.

\section{Preliminaries}
\label{sec:1}
In what follows we use the notations $f\asymp g$ and $f\lesssim g$ in $D$, which means that there exists constants  $C,C_1,C_2>0$ such that $C_1 g\leq f\leq C_2 g$ and $f\leq C g$ in $D$, respectively. Such constants may change line by line. We also use the notation $\Omega=\norm{x}^\beta t^{-\alpha}$ for $x\in \RR^d$ and $t>0$.

\medbreak

The following result summarizes the two-sided estimates for $Z$. For its proof see  \cite[Theorem 2]{JK19}. 

\begin{proposition}
	\label{cotasZ}
	Let $\alpha\in (0,1)$ and $\beta\in (0,2)$. Assume the hypothesis $(\mathcal{H}_1)$ holds. Then there exists a positive constant $C$ such that for $(t,x)\in(0,\infty)\times\mathbb{R}^d$ the following two-sided estimates for $Z$ hold. For $\Omega\leq 1$,
		\begin{align*}
		Z(t,x)\asymp C t^{-\frac{\alpha d}{\beta}}~~~~~~~~~~~~~~~&\quad\text{if}\quad d<\beta,\\
		Z(t,x)\asymp C t^{-\alpha}(|\log(\Omega)|+1) &\quad\text{if}\quad d=\beta,\\
		Z(t,x)\asymp C t^{-\frac{\alpha d}{\beta}}\Omega^{1-\frac{d}{\beta}}~~~~~~~~ &\quad\text{if}\quad d>\beta.
		\end{align*}
		
		For $\Omega\geq 1$,
		\begin{equation*}
		Z(t,x)\asymp C t^{-\frac{\alpha d}{\beta}}\Omega^{-1-\frac{d}{\beta}}.
		\end{equation*}
\end{proposition}

In the same way we have derived the two-sided estimates for $Y$.

\begin{proposition}
	\label{cotasY}
	Under the same assumptions as Proposition \ref{cotasZ}, the following two-sided estimates for $Y$ hold.
	For $\Omega\leq 1$,
		\begin{align*}
		Y(t,x)\asymp C t^{-\frac{\alpha d}{\beta}+\alpha-1}~~~~~~~~~~~~\,&\quad \text{if}\quad d<2\beta,\\
		Y(t,x)\asymp C t^{-\alpha-1}(|\log(\Omega)|+1)~ &\quad \text{if}\quad d=2\beta, \\
		Y(t,x)\asymp C t^{-\frac{\alpha d}{\beta}+\alpha-1}\Omega^{2-\frac{d}{\beta}}~~~~~~&\quad \text{if}\quad d>2\beta.
		\end{align*}
		
		For $\Omega\geq 1$,
		\begin{equation*}
		Y(t,x)\asymp C t^{-\frac{\alpha d}{\beta}+\alpha-1}\Omega^{-1-\frac{d}{\beta}}.
		\end{equation*}
\end{proposition}
\begin{proof}
	The assertions follow from straightforward computations made in the proof of the previous estimates for $Z$, in \cite[Theorem 2]{JK19}. There, the authors used the fact that the asymptotic behaviour of $G_\alpha$ is the same as for the density $w_\alpha$ given in \cite[Proposition 1]{JK19} (with the skewness of the distribution that equals to $0$) by
\[
		w_\alpha(\tau) \sim C
		\begin{cases}
			\tau^{-1-\alpha} &\quad \text{as}\quad \tau\rightarrow\infty,\\
			f_\alpha(\tau):=\tau^{-\frac{2-\alpha}{2(1-\alpha)}}e^{-c_\alpha\tau^{-\frac{\alpha}{1-\alpha}}} &\quad \text{as}\quad \tau\rightarrow 0,\end{cases}		
\]	
where $c_\alpha=(1-\alpha)\alpha^{\frac{\alpha}{1-\alpha}}$. See e.g., \cite[Proposition 2.4.1]{Kol19} and \cite[Theorem 2.5.2]{Zol86} for more details.

Keeping this in mind, we see that a difference between the functions $Z$ and $Y$, given by \eqref{Z:G} and \eqref{Y:G} respectively, is the factor $s^{-1}$ inside the improper Riemann integral of $Z$.  Thus, we only need to check the corresponding two-sided estimates for
\[
\displaystyle\int_0^\infty G(t^\alpha s,x)s^{-\frac{1}{\alpha}}G_\alpha(1,s^{-\frac{1}{\alpha}})ds,
\]
which can be written equivalently as
\[
\displaystyle\int_0^\infty G(t^\alpha s,x) s^{-1-\frac{1}{\alpha}}G_\alpha(1,s^{-\frac{1}{\alpha}})s\,ds\asymp I_1+I_2.
\]
Here, similar to the integrals that are used in \cite[expression (33)]{JK19},
\begin{equation}
\label{I_1}
I_1:=\int_0^1 \min\left(t^{-\frac{\alpha d}{\beta}}\Omega^{-1-\frac{d}{\beta}}s,t^{-\frac{\alpha d}{\beta}}s^{-\frac{d}{\beta}}\right) s\, ds
\end{equation}
and
\begin{equation}
\label{I_2}
I_2:=\int_1^\infty \min\left(t^{-\frac{\alpha d}{\beta}}\Omega^{-1-\frac{d}{\beta}}s,t^{-\frac{\alpha d}{\beta}}s^{-\frac{d}{\beta}}\right) s^{-1-\frac{1}{\alpha}}f_\alpha(s^{-\frac{1}{\alpha}}) s\, ds,
\end{equation}
where
\[
\min\left(t^{-\frac{\alpha d}{\beta}}\Omega^{-1-\frac{d}{\beta}}s,t^{-\frac{\alpha d}{\beta}}s^{-\frac{d}{\beta}}\right)=\begin{cases}t^{-\frac{\alpha d}{\beta}}\Omega^{-1-\frac{d}{\beta}}s, & \text{for } s<\Omega,\\t^{-\frac{\alpha d}{\beta}}s^{-\frac{d}{\beta}}, & \text{for } s\geq\Omega,\end{cases}
\]
as in \cite[expression (32)]{JK19}. Next, we need to analyse the two-sided estimates for $I_j$, $j=1,2$. The case $\Omega\leq 1$ yields
\begin{align*}
I_1&=t^{-\frac{\alpha d}{\beta}}\Omega^{-1-\frac{d}{\beta}}\int_0^\Omega s^2 ds
+t^{-\frac{\alpha d}{\beta}}\int_\Omega^1 s^{-\frac{d}{\beta}} s\, ds\\
&=\frac{1}{3}t^{-\frac{\alpha d}{\beta}}\Omega^{2-\frac{d}{\beta}} +t^{-\frac{\alpha d}{\beta}}\int_\Omega^1 s^{1-\frac{d}{\beta}} ds.
\end{align*}
The last integral requires the sub-cases $d<2\beta$, $d=2\beta$ and $d>2\beta$:
\[
t^{-\frac{\alpha d}{\beta}}\int_\Omega^1 s^{1-\frac{d}{\beta}}ds=\begin{cases}t^{-\frac{\alpha d}{\beta}}\frac{1}{2-\frac{d}{\beta}}\left(1-\Omega^{2-\frac{d}{\beta}}\right), & \text{for } d<2\beta,\\
t^{-2\alpha}|log(\Omega)|, & \text{for } d=2\beta,\\
t^{-\frac{\alpha d}{\beta}}\frac{1}{\frac{d}{\beta}-2}\left(\Omega^{2-\frac{d}{\beta}}-1\right), & \text{for } d>2\beta.\end{cases}
\]
For $I_2$, since $\Omega\leq 1$, we have that
\begin{align*}
I_2&=t^{-\frac{\alpha d}{\beta}}\int_1^\infty s^{-\frac{d}{\beta}} s^{-1-\frac{1}{\alpha}}f_\alpha(s^{-\frac{1}{\alpha}}) s\, ds\\
&=t^{-\frac{\alpha d}{\beta}}\int_1^\infty s^{-\frac{d}{\beta}-\frac{1}{\alpha}}s^{\frac{2-\alpha}{2\alpha(1-\alpha)}}e^{-(1-\alpha)\alpha^{\frac{\alpha}{1-\alpha}}s^{\frac{1}{1-\alpha}}} ds\\
&=t^{-\frac{\alpha d}{\beta}}\int_1^\infty s^{-\frac{d}{\beta}+\frac{1}{2(1-\alpha)}}e^{-(1-\alpha)\alpha^{\frac{\alpha}{1-\alpha}}s^{\frac{1}{1-\alpha}}} ds\\
&=C t^{-\frac{\alpha d}{\beta}}.
\end{align*}
We point out that the improper integral is convergent due to the Laplace method for integrals (see e.g., \cite[(A1)]{JK19}). Therefore, if $d<2\beta$ we find that 
\[
C t^{-\frac{\alpha d}{\beta}}=I_2\leq I_1+I_2=\frac{1}{3}t^{-\frac{\alpha d}{\beta}}\Omega^{2-\frac{d}{\beta}}
+t^{-\frac{\alpha d}{\beta}}\frac{1}{2-\frac{d}{\beta}}\left(1-\Omega^{2-\frac{d}{\beta}}\right)\lesssim t^{-\frac{\alpha d}{\beta}},
\]
if $d=2\beta$ we obtain
\[
I_1+I_2=\frac{1}{3}t^{-2\alpha}+t^{-2\alpha}|log(\Omega)| +C t^{-2\alpha}
\]
and $d>2\beta$ implies that
\[
\frac{1}{3}t^{-\frac{\alpha d}{\beta}}\Omega^{2-\frac{d}{\beta}}\leq I_1\leq I_1+I_2=\frac{1}{3}t^{-\frac{\alpha d}{\beta}}\Omega^{2-\frac{d}{\beta}}+t^{-\frac{\alpha d}{\beta}}\frac{1}{\frac{d}{\beta}-2}\left(\Omega^{2-\frac{d}{\beta}}-1\right)\lesssim t^{-\frac{\alpha d}{\beta}}\Omega^{2-\frac{d}{\beta}}.
\]
Since the additional factor $t^{\alpha-1}$ is a constant for the integral of $Y$, the estimates hold for $\Omega\leq 1$. Now, the case $\Omega\geq 1$ implies that
\begin{align*}
I_1&=t^{-\frac{\alpha d}{\beta}}\Omega^{-1-\frac{d}{\beta}}\int_0^1 s^2 ds\\
&=\frac{1}{3}t^{-\frac{\alpha d}{\beta}}\Omega^{-1-\frac{d}{\beta}}
\end{align*}
and
\[
I_2=t^{-\frac{\alpha d}{\beta}}\Omega^{-1-\frac{d}{\beta}}\int_1^\Omega s^{1-\frac{1}{\alpha}}f_\alpha(s^{-\frac{1}{\alpha}})ds + t^{-\frac{\alpha d}{\beta}}\int_\Omega^\infty s^{-\frac{d}{\beta}-\frac{1}{\alpha}}f_\alpha(s^{-\frac{1}{\alpha}}) ds.
\]
We see that 
\begin{align*}
I_2&\leq t^{-\frac{\alpha d}{\beta}}\Omega^{-1-\frac{d}{\beta}}\int_1^\Omega s^{1-\frac{1}{\alpha}}f_\alpha(s^{-\frac{1}{\alpha}})ds+t^{-\frac{\alpha d}{\beta}}\Omega^{-1-\frac{d}{\beta}}\int_\Omega^\infty s^{1-\frac{1}{\alpha}}f_\alpha(s^{-\frac{1}{\alpha}}) ds\\
&=t^{-\frac{\alpha d}{\beta}}\Omega^{-1-\frac{d}{\beta}}\int_1^\infty s^{1-\frac{1}{\alpha}}f_\alpha(s^{-\frac{1}{\alpha}})ds\\
&=C_1t^{-\frac{\alpha d}{\beta}}\Omega^{-1-\frac{d}{\beta}}.
\end{align*}
On the other hand,
\begin{align*}
I_2&\geq t^{-\frac{\alpha d}{\beta}}\Omega^{-1-\frac{d}{\beta}}\int_1^\Omega s^{-\frac{d}{\beta}-\frac{1}{\alpha}}f_\alpha(s^{-\frac{1}{\alpha}})ds + t^{-\frac{\alpha d}{\beta}}\Omega^{-1-\frac{d}{\beta}}\int_\Omega^\infty s^{-\frac{d}{\beta}-\frac{1}{\alpha}}f_\alpha(s^{-\frac{1}{\alpha}}) ds\\
&=t^{-\frac{\alpha d}{\beta}}\Omega^{-1-\frac{d}{\beta}}\int_1^\infty s^{-\frac{d}{\beta}-\frac{1}{\alpha}}f_\alpha(s^{-\frac{1}{\alpha}})ds \\
&=C_2t^{-\frac{\alpha d}{\beta}}\Omega^{-1-\frac{d}{\beta}}.
\end{align*}
These bounds show that $I_1+I_2\asymp t^{-\frac{\alpha d}{\beta}}\Omega^{-1-\frac{d}{\beta}}$ for $\Omega\geq 1$. The factor $t^{\alpha-1}$ completes the proof.
\end{proof}

\begin{remark}
We note a singularity at the origin with respect to the spatial variable for $Z$, whenever $d\geq\beta$, and for $Y$ whenever $d\geq 2\beta$. It is well known that this type of singularities occurs in the equations of fractional evolution in time, even if $\beta=2$ and $\omega_\mu\equiv 1$.
\end{remark}
In order to formulate our results, we also recall the following properties of the fundamental solutions $Z$ and $Y$, lemmata \ref{cotasDeltaZ}-\ref{cotasDeltaEspacialY} (see proofs in \cite[Section 2]{SV22}).
\begin{lemma}
	\label{cotasDeltaZ}
	Under the same assumptions as Proposition \ref{cotasZ}, there exists a positive constant $C$ for all $t_1,t_2>0$ and $x\in\RR^d$, such that there exists $t_c>0$, between $t_1$ and $t_2$, and the following estimates for $Z$ hold with $\Omega_c=\norm{x}^\beta t_c^{-\alpha}$. For $\Omega_c\leq 1$,		
	\begin{equation*}
		|Z(t_1,x)-Z(t_2,x)|\leq C |t_1-t_2|
		\begin{cases}
			t_c^{-\frac{\alpha d}{\beta}-1} &\quad \text{if}\quad d<\beta,\\
			t_c^{-\alpha-1}(|\log(\Omega_c)|+1) &\quad \text{if}\quad d=\beta,\\
			t_c^{-\frac{\alpha d}{\beta}-1}\Omega_c^{1-\frac{d}{\beta}}&\quad \text{if}\quad d>\beta,		
		\end{cases}
	\end{equation*}
	and for $\Omega_c\geq 1$,
	\[
	|Z(t_1,x)-Z(t_2,x)|\leq C |t_1-t_2| t_c^{-\frac{\alpha d}{\beta}-1}\Omega_c^{-1-\frac{d}{\beta}}.
	\]
\end{lemma}	

\begin{lemma}
	\label{cotasDeltaY}
	Under the same assumptions as Proposition \ref{cotasZ}, there exists a positive constant $C$ for all $t_1,t_2>0$ and $x\in\RR^d$, such that there exists $t_c>0$, between $t_1$ and $t_2$, and the following estimates for $Y$ hold with $\Omega_c=\norm{x}^\beta t_c^{-\alpha}$. For $\Omega_c\leq 1$,
	\[
	|Y(t_1,x)-Y(t_2,x)|\leq C |t_1-t_2|	
	\begin{cases}
		t_c^{-\frac{\alpha d}{\beta}+\alpha-2} &\quad \text{if}\quad d<2\beta,\\
		t_c^{-\alpha-2}(|\log(\Omega_c)|+1) &\quad \text{if}\quad d=2\beta,\\
		t_c^{-\frac{\alpha d}{\beta}+\alpha-2}\Omega_c^{2-\frac{d}{\beta}}&\quad \text{if}\quad d>2\beta,		
	\end{cases}
	\]
	and for $\Omega_c\geq 1$,
	\[
	|Y(t_1,x)-Y(t_2,x)|\leq C |t_1-t_2| t_c^{-\frac{\alpha d}{\beta}+\alpha-2}\Omega_c^{-1-\frac{d}{\beta}}.
	\]	
\end{lemma}

\begin{lemma}
	\label{cotasDeltaEspacialZ}
Let $\alpha\in (0,1)$ and $\beta\in (0,2)$. Assume the hypothesis $(\mathcal{H}_2)$ holds. Then there exists a positive constant $C$ for all $t>0$ and $x_1,x_2\in\RR^d$, such that there exists $\zeta$ in the open segment connecting $x_1$ and $x_2$, and the following estimates for $Z$ hold with $\Omega_\zeta=\norm{\zeta}^\beta t^{-\alpha}$. For $\Omega_\zeta\leq 1$,		
	\[
	|Z(t,x_1)-Z(t,x_2)|\leq C \norm{x_1-x_2}t^{-\frac{\alpha (d+1)}{\beta}}\Omega_\zeta^{1-\frac{d+1}{\beta}}
	\]
	and for $\Omega_\zeta\geq 1$,
	\[
	|Z(t,x_1)-Z(t,x_2)|\leq C \norm{x_1-x_2} t^{-\frac{\alpha (d+1)}{\beta}}\Omega_\zeta^{-1-\frac{d+1}{\beta}}.
	\]
\end{lemma}

\begin{lemma}
	\label{cotasDeltaEspacialY}
	Under the same assumptions as Lemma \ref{cotasDeltaEspacialZ}, then there exists a positive constant $C$ for all $t>0$ and $x_1,x_2\in\RR^d$, such that there exists $\zeta$ in the open segment connecting $x_1$ and $x_2$, and the following estimates for $Y$ hold with $\Omega_\zeta=\norm{\zeta}^\beta t^{-\alpha}$. For $\Omega_\zeta\leq 1$,
	\begin{equation*}
		|Y(t,x_1)-Y(t,x_2)|\leq C \norm{x_1-x_2}
		\begin{cases}
			t^{-\frac{\alpha (d+1)}{\beta}+\alpha-1} &\quad \text{if}\quad d+1<2\beta,\\
			t^{-\alpha-1}(|\log(\Omega_\zeta)|+1) &\quad \text{if}\quad d+1=2\beta,\\
			t^{-\frac{\alpha (d+1)}{\beta}+\alpha-1}\Omega_\zeta^{2-\frac{d+1}{\beta}}&\quad \text{if}\quad d+1>2\beta,		
		\end{cases}
	\end{equation*}
	and for $\Omega_\zeta\geq 1$,
	\[
	|Y(t,x_1)-Y(t,x_2)|\leq C \norm{x_1-x_2} t^{-\frac{\alpha (d+1)}{\beta}+\alpha-1}\Omega_\zeta^{-1-\frac{d+1}{\beta}}.
	\]	
\end{lemma}

\begin{lemma}
\label{NormaGradienteY}
Under the same assumptions as Lemma \ref{cotasDeltaEspacialZ}, then there exists a positive constant $C$ for all $t>0$ and $x_1,x_2\in\RR^d$, such that the estimate
\begin{equation}
\label{cotasNablaY}
\norm{Y(t,\cdot-x_1)-Y(t,\cdot-x_2)}_q\leq C \norm{x_1-x_2}\norm{\nabla Y(t,\cdot)}_q\lesssim \norm{x_1-x_2} t^{-\frac{\alpha d}{\beta}\left(1-\frac{1}{q}\right)-\frac{\alpha}{\beta}+\alpha-1}	
\end{equation} 	
is true for $1\leq q<\kappa'$, where $\kappa':=\begin{cases}\frac{d}{d+1-2\beta}, & d+1>2\beta\text{ and }\beta>\frac{1}{2},\\ \infty, & d+1\leq 2\beta. \end{cases}$

In the case of $d+1<2\beta$, \eqref{cotasNablaY} remains true for $q=\infty$.
\end{lemma}

\begin{proof}
It follows from the same arguments as in \cite[Lemma 6.1]{SV22} but using the bounds given in Lemma \ref{cotasDeltaEspacialY}.
\end{proof}

It is worth mentioning that all these estimates have been thoroughly investigated using the \textit{Zolotarev-Pollard formula for Mittag-Leffler functions} $E_\alpha$, which is valid for the case $0<\alpha<1$ (see \cite[Section 2]{JK19} and \cite[Proposition 8.1.1]{Kol19}). To our knowledge this type of representation has not been explored explicitly in the literature for the case $\alpha >1$, however, we refer the reader to \cite{Baz18} and \cite{BSD20} for the study of evolution equations with a Caputo fractional derivative of order $1<\alpha<2$. 

\begin{definition}
\label{defSolution}
Let $\alpha\in (0,1)$, $\beta\in (0,2)$ and $\gamma>1$. Assume the hypothesis $(\mathcal{H}_1)$ holds. Suppose that $1<p<\infty$ and that $u_0\in L_p(\RR^d)$ is a non-negative function. A function $u$ is called a \textbf{local solution} of \eqref{general}, if there exists $T>0$ such that  
\begin{enumerate}[(i)]
\item $u\in C([0,T];L_p(\RR^d))\cap L_{\infty}((0,T)\times\RR^d)$,
\item $u$ satisfies \eqref{general} in $[0,T]$.
\end{enumerate}
A function $u$ is called a \textbf{global solution} of \eqref{general} if $(i)$-$(ii)$ are satisfied for any $T>0$. We say that $u$ is a \textbf{mild solution} of \eqref{general} if $u\in C([0,T];L_p(\RR^d))\cap L_{\infty}((0,T)\times\RR^d)$ and it satisfies the integral equation
\[
u(t,x) = \int_{\RR^d}Z(t,x-y)u_0(y)dy + \int_0^t \int_{\RR^d}Y(t-s,x-y)|u(s,y)|^{\gamma-1}u(s,y)dyds
\]
for all $x\in\RR^d$ and $0\leq t< T$.
\end{definition}
At this point, we mention that problems like \eqref{general} have been studied in \citep[Section 5]{SV22}. Under suitable conditions on $\alpha$, $\beta$, $\gamma$ and $p$, together with other parameters, the authors find positive, local and global solutions.

\section{Representation of solution in its integral form}
\label{sec:2}
In this section we analyse the conditions under which a local solution $u$ of \eqref{general}, in the sense of Definition \ref{defSolution}, can be represented as 
\begin{equation}
\label{integral}
u(t,x) = \int_{\RR^d}Z(t,x-y)u_0(y)dy + \int_0^t \int_{\RR^d}Y(t-s,x-y)|u(s,y)|^{\gamma-1}u(s,y)dyds
\end{equation}
for all $x\in\RR^d$ and $0\leq t< T$.

Although the subordination principle employed here follows directly from \cite[Chapter 3]{Baz01}, for instance, the point we want to emphasize is the relation \eqref{relacionYZ}, between the fundamental solutions $Z$ and $Y$ in the context of non-Gaussian process, which leads to the main result of this section.

First, we recall that the symbol $\psi(\xi)$ is a continuous and negative definite function. Thereby, from \cite[Example 4.6.29]{Jac01} we know that $(-\Psi_{\beta}(-i\nabla), C_0^{\infty}(\RR^d))$ satisfies, for any $1<p<\infty$, the Dirichlet condition 
\[
\int_{\RR^d}\left(-\Psi_{\beta}(-i\nabla)f\right)(x)\left((f-1)^+\right)^{p-1}(x)dx\leq 0,\quad f\in C_0^{\infty}(\RR^d),
\]
and consequently it is $L_p(\RR^d)$-dissipative (\cite[Propositions 4.6.11 and 4.6.12]{Jac01}). In fact, the density of $C_0^{\infty}(\RR^d)$ in $L_p(\RR^d)$ implies that $(-\Psi_{\beta}(-i\nabla), C_0^{\infty}(\RR^d))$ is closable and its closure $(A, D(A))$ generates a sub-Markovian semigroup $\{T_t\}_{t\geq 0}$ on $L_p(\RR^d)$ which is a strongly continuous contraction semigroup (\cite[Lemma 4.1.36, Theorems 4.1.33 and 4.6.17, Definition 4.1.6]{Jac01}). Besides, $A$ is densely defined on $L_p(\RR^d)$ (\cite[Corollary 4.1.15]{Jac01}). On the other hand, it is well known that $g_\alpha$ is a completely positive function and belongs to $L_{1,loc}(\RR^+)$.
 
We denote $u(t)=u(t,\cdot)$ and $|u|^{\gamma-1}(t)=\abs{u(t,\cdot)}^{\gamma-1}$. Since $u$ and $u_0$ satisfy Definition \ref{defSolution}, if $u_0\in L_\infty(\RR^d)$ we observe that $g_{\alpha}*|u|^{\gamma-1}u(t)\in L_p(\RR^d)$ for $0\leq t<T$. %\textcolor{blue}{If additionally the norms $\norm{Au(t)}_{L_p(\RR^d)}$ are bounded on compact intervals of $t$, then} the
Equation \eqref{general} can be written as the Volterra equation
\begin{equation}
\label{Volterra}
u(t)=u_0+g_{\alpha}*|u|^{\gamma-1}u(t)+g_{\alpha}*Au(t), \quad 0<t<T,
\end{equation}   
which admits a resolvent $\{S(t)\}_{t\geq 0}$ in $L_p(\RR^d)$ (\cite[Theorems 4.1 and 4.2]{Pru93}). From \cite[Corollary 4.5]{Pru93} we have that
\[
S(t)=-\int_0^\infty T_\tau \; w(t;d\tau),\quad t>0,
\]
where $w$ is the propagation function associated with $g_\alpha$. In order to describe this resolvent, we use the representation
\[
T_tf(\cdot)=\int_{\RR^d} G(t,\cdot-y)f(y)dy, \quad f\in D(A),
\]
the function $G$ being the fundamental solution of the problem \eqref{HomogeneaOrden1} (see \cite[Section 1.2 Theorem 2.4 (c) and Section 4.1 Theorem 1.3 ]{Paz83}). For $v\in D(A)$ we see that
\begin{align*}
S(t)v&=-\int_0^\infty T_\tau v\; w(t;d\tau)\\
&=-\int_0^\infty G(t,\cdot)\star v\; w(t;d\tau)
\end{align*}
and using the Fourier transform we obtain
\begin{align*}
\mathcal{F}(S(t)v)&=-\int_0^\infty e^{-\tau\psi(\xi)} \widehat{v}\; w(t;d\tau)\\
&=s(t,\psi(\xi))\widehat{v}\\
&=\widehat{Z}(t,\xi)\widehat{v}
\end{align*}
with a kernel $s$ that comes via scalar Volterra equations (see \cite[Proposition 4.9]{Pru93}, \cite[Sections 2 and 3]{PV18}). This implies that
\[
S(t)v=Z(t,\cdot)\star v
\]
and the boundedness of $S(t)$ leads to an extension to all of $L_p(\RR^d)$.

Let $0<t<T$. If $u(s)\in D(A)$, $0\leq s\leq t$, identity \eqref{Volterra} and \cite[Proposition 1.1, Definition 1.3]{Pru93} yield
\begin{align*}
1*u(t)&=\int_0^t u(s)ds\\
&=\int_0^t \left(S(t-s)u(s)-A(g_\alpha * S)(t-s)u(s)\right) ds\\
&=\int_0^t S(t-s)u(s)ds-\int_0^t(g_\alpha * S)(t-s)Au(s)ds\\
&=\int_0^t S(s)u(t-s)ds-\int_0^t S(s)(g_\alpha * Au)(t-s) ds\\
&=\int_0^t S(s)\left(u(t-s)-(g_\alpha * Au)(t-s)\right) ds\\
&=\int_0^t S(s)\left(u_0+g_{\alpha}*|u|^{\gamma-1}u(t-s)\right) ds
\end{align*}
and thus we get the \textit{variation of parameters formula} for \eqref{Volterra} given by
\[
u(t)=\frac{d}{dt}\int_0^t S(s)\left(u_0+g_{\alpha}*|u|^{\gamma-1}u\right)(t-s) ds.
\] 
We note that $$\frac{d}{dt}\int_0^t S(s)u_0ds=S(t)u_0=Z(t,\cdot)\star u_0.$$ By proceeding as in the proof of \cite[Lemma 5.1]{SV22}, but working with the $L_p(\RR^d)$ space, using the relation \eqref{relacionYZ} and the fact that $\underset{0\leq t<T}{\sup}{\left\||u|^{\gamma-1}u(t)\right\|}_\infty<\infty$, we show that
$$\frac{d}{dt}\int_0^t S(s)\left(g_{\alpha}*|u|^{\gamma-1}u\right)(t-s)ds=\int_0^t Y(t-s,\cdot)\star|u|^{\gamma-1}u(s,\cdot)ds.$$

\begin{theorem}
\label{localrepresentation}
Let $\alpha\in (0,1)$ and $\beta\in (0,2)$. Assume the hypothesis $(\mathcal{H}_1)$ holds. Let $\gamma>1$ and suppose that $1<p<\infty$. Let $u_0\in D(A)\cap L_\infty(\RR^d)$ be a non-negative function. If $u$ is a local solution in the sense of Definition \ref{defSolution} for some $T>0$ and $u(t)\in D(A)$ for all $0\leq t<T$, then $u$ admits the representation $\eqref{integral}$.
\end{theorem}

\section{Continuity and non-negativeness of solution in $[0,T)\times \RR^d$}
\label{sec:3}
Let $u$ be a local solution of \eqref{general}. In this section we show that $u$ is a continuous and non-negative function on $[0,T)\times \RR^d$, for some $T>0$. For this purpose, the representation \eqref{integral} obtained in the previous section is particularly important. Besides, we need the following technical result.
\begin{lemma}
\label{IdentidadAproximada}
Let $d\in\mathbb{N}$, $\alpha\in (0,1)$ and $\beta\in (0,2)$. Assume the hypothesis $(\mathcal{H}_1)$ holds. If $f$ is a continuous and bounded function on $\RR^d$, then $Z(t,\cdot)\star f\rightarrow f$ uniformly on compact sets whenever $t\rightarrow 0$.  
\end{lemma}
\begin{proof}
From \cite[Lemma 2.12 and Formula (2.20)]{SV22} we know that $g(x):=Z(1,x)$, $x\in\RR^d$, satisfies all assumptions of \cite[Theorem 1.6]{Fol83} with $\epsilon=t^{\frac{\alpha}{\beta}}$.
\end{proof}

In what follows we use the parameter $\kappa:=\begin{cases}\frac{d}{\beta}, & d>\beta,\\ 1, & otherwise \end{cases}$ which sets a condition on $p$ for the existence of some $q\geq 1$ such that
\[
\frac{1}{p}+\frac{1}{q}=1
\]
and the $L_q$-norm for $Y(t,\cdot)$, $t>0$, is reached. Indeed, by choosing $\kappa<p<\infty$ we obtain that $1<q<\infty$ whenever $\kappa=1$ and
$1<q<\frac{d}{d-\beta}$ whenever $\kappa=\frac{d}{\beta}$. This implies that $q<\kappa_2$, with $\kappa_2$ as in \cite[Theorem 2.10]{SV22}.

\begin{theorem}
\label{continuidad}
Let $\alpha\in (0,1)$ and $\beta\in (1,2)$. Assume the hypothesis $(\mathcal{H}_2)$ holds. Let $\gamma>1$ and suppose that $\max\left(1,\kappa\right)<p<\infty$. Let $u_0\in D(A)\cap L_\infty(\RR^d)\cap C(\RR^d)$ be a non-negative function. If $u$ is a local solution in the sense of Definition \ref{defSolution} for some $T>0$ and $u(t)\in D(A)$ for all $0\leq t<T$, then $u\in C([0,T)\times \RR^d)$.  
\end{theorem}
\begin{proof}
From Theorem \ref{localrepresentation} it follows that the local solution $u$ has the form
\[
u(t,x) = \int_{\RR^d}Z(t,x-y)u_0(y)dy + \int_0^t \int_{\RR^d}Y(t-s,x-y)|u|^{\gamma-1}u(s,y)dyds,\quad x\in\RR^d,\; 0\leq t< T.
\]
We define
\[
u_1(t,x):= \int_{\RR^d}Z(t,x-y)u_0(y)dy 
\]
and
\[
u_2(t,x):=\int_0^t \int_{\RR^d}Y(t-s,x-y)|u|^{\gamma-1}u(s,y)dyds.
\]
We shall show that for all $\epsilon>0$, there exists $\delta>0$ such that
\[
|u_j(t,x)-u_j(t_0,x_0)|< \epsilon, \forall (t,x)\in B((t_0,x_0),\delta)\subset [0,T)\times \RR^d,
\]
for $j\in\{1,2\}$.

Let $x_0\in\RR^d$ and $0<t_0<T$. We suppose $t_0<t<T$ without loss of generality. For $u_1$ we see that
\begin{align*}
|u_1(t,x)-u_1(t_0,x_0)|&\leq\int_{\RR^d}|Z(t,x-y)-Z(t_0,x_0-y)|u_0(y)dy\\
&\leq\int_{\RR^d}|Z(t,x-y)-Z(t_0,x-y)|u_0(y)dy\\
&~~~+\int_{\RR^d}|Z(t_0,x-y)-Z(t_0,x_0-y)|u_0(y)dy\\
&\lesssim \norm{u_0}_\infty\int_{\RR^d}|Z(t,x-y)-Z(t_0,x-y)|dy\\
&~~~+\norm{u_0}_\infty\int_{\RR^d}|Z(t_0,x-y)-Z(t_0,x_0-y)|dy\\
&\lesssim \norm{u_0}_\infty|t-t_0|t_0^{-1}+\norm{u_0}_\infty\norm{x-x_0}t_0^{-\frac{\alpha}{\beta}},
\end{align*}  
where the last estimates follow from \cite[Theorem 2.13 and Lemma 6.1]{SV22}, respectively. Thus,
\[
|u_1(t,x)-u_1(t_0,x_0)|\lesssim |t-t_0|t_0^{-1}+\norm{x-x_0}t_0^{-\frac{\alpha}{\beta}}
\]
and we can take a ball in $\RR^d$ of radius $C^{-1}\epsilon t_0^{\frac{\alpha}{\beta}}$ centered at $x_0$, and an interval in $[0,T)$ of radius $C^{-1}\epsilon t_0$ centered at $t_0$, where $C$ is the constant of the estimate.

For the continuity of $u_1$ in $(0,x_0)$ we have that
\begin{align*}
|u_1(t,x)-u_1(0,x_0)|&=|u_1(t,x)-u_0(x_0)|\\
&=|u_1(t,x)-u_0(x)+u_0(x)-u_0(x_0)|\\
&\leq |u_1(t,x)-u_0(x)|+|u_0(x)-u_0(x_0)|\\
&=\left|\int_{\RR^d}Z(t,x-y)u_0(y)dy-u_0(x)\right|+|u_0(x)-u_0(x_0)|.
\end{align*}  
We note that, by Lemma \ref{IdentidadAproximada}, the continuity and boundedness of $u_0$ imply the uniform limit on compact subsets of $\RR^d$ for the first term as $t\rightarrow 0$. By choosing a sufficiently small $\delta$ we get the desired result.

Next, we analyse the continuity of $u_2$. We see that
\begin{align*}
|u_2(t,x)|&\leq\int_0^t \int_{\RR^d}Y(t-s,x-y)|u(s,y)|^{\gamma}dyds\\
&\leq\sup_{0\leq s\leq t}\norm{u(s)}_\infty^\gamma\int_0^t \int_{\RR^d}Y(t-s,x-y)dyds\\
&\leq\sup_{0\leq s\leq t}\norm{u(s)}_\infty^\gamma\int_0^t \frac{(t-s)^{\alpha-1}}{\Gamma(\alpha)}ds\\
&\leq\sup_{0\leq s\leq t}\norm{u(s)}_\infty^\gamma\frac{t^{\alpha}}{\Gamma(\alpha+1)}.
\end{align*}
This proves that
\[
\lim_{t\rightarrow 0}u_2(t,x)=0
\]
uniformly on $\RR^d$.

Now, let $x_0\in\RR^d$ and $0<t_0<T$. Again, we suppose $t_0<t<T$ without loss of generality. We find that
\begin{align*}
&|u_2(t,x)-u_2(t_0,x_0)|\\
&\leq |u_2(t,x)-u_2(t_0,x)|+|u_2(t_0,x)-u_2(t_0,x_0)|\\
&\leq\int_0^{t_0}\int_{\RR^d}Y(s,x-y)\left||u|^{\gamma-1}u(t-s,y)-|u|^{\gamma-1}u(t_0-s,y)\right|dy ds\\
&~~~+\int_{t_0}^{t}\int_{\RR^d}Y(s,x-y)|u(t-s,y)|^\gamma dy ds\\
&~~~+\int_0^{t_0}\int_{\RR^d}|Y(t_0-s,x-y)-Y(t_0-s,x_0-y)||u(s,y)|^\gamma dy ds\\
&\lesssim\gamma \sup_{0\leq s\leq t}\norm{u(s)}_\infty^{\gamma-1}\int_0^{t_0} \int_{\RR^d}Y(s,x-y)|u(t-s,y)-u(t_0-s,y)|dy ds\\
&~~~+\sup_{0\leq s\leq t}\norm{u(s)}_\infty^\gamma\int_{t_0}^{t}\int_{\RR^d}Y(s,x-y)dy ds\\
&~~~+\sup_{0\leq s\leq t}\norm{u(s)}_\infty^\gamma\int_0^{t_0}\int_{\RR^d}|Y(t_0-s,x-y)-Y(t_0-s,x_0-y)| dy ds\\
&\lesssim\gamma \sup_{0\leq s\leq t}\norm{u(s)}_\infty^{\gamma-1}\int_0^{t_0}\left\|Y(s,\cdot)\star|u(t-s)-u(t_0-s)|\right\|_{\infty}ds\\
&~~~+\sup_{0\leq s\leq t}\norm{u(s)}_\infty^\gamma\int_{t_0}^{t}s^{\alpha-1} ds\\
&~~~+\sup_{0\leq s\leq t}\norm{u(s)}_\infty^\gamma\int_0^{t_0}\norm{x-x_0}(t_0-s)^{-\frac{\alpha}{\beta}+\alpha-1} ds,
\end{align*}
where the last integral is estimated by Lemma \ref{NormaGradienteY}. For estimating the first term, we use the continuity of $u$ with respect to the norm topology on $L_p(\RR^d)$ and Young's convolution inequality, i.e.,
\begin{align*}
\int_0^{t_0}\left\|Y(s,\cdot)\star|u(t-s)-u(t_0-s)|\right\|_{\infty}ds&\lesssim\int_0^{t_0}\norm{Y(s,\cdot)}_q\norm{u(t-s)-u(t_0-s)}_p\;ds\\
&\lesssim \epsilon\int_0^{t_0}s^{-\frac{\alpha d}{\beta p}+\alpha-1}ds.
\end{align*}  
Thus,
\[
|u_2(t,x)-u_2(t_0,x_0)|\lesssim \epsilon t_0^{\alpha-\frac{\alpha d}{\beta p}}+(t^{\alpha}-t_0^{\alpha})+\norm{x-x_0}t_0^{\alpha-\frac{\alpha}{\beta}}.
\]
\end{proof}
The second result of this section is the following.
\begin{theorem}
\label{positividad}
Let $\alpha\in (0,1)$ and $\beta\in (0,2)$. Assume the hypothesis $(\mathcal{H}_1)$ holds. Let $\gamma>1$ and suppose that $1<p<\infty$. Let $u_0\in D(A)\cap L_\infty(\RR^d)$ be a non-negative function. If $u$ is a local solution in the sense of Definition \ref{defSolution} for some $T>0$ and $u(t)\in D(A)$ for all $0\leq t<T$, then there exists $0<T^*\leq T$ such that $u$ is non-negative in $[0,T^*)\times \RR^d$.  
\end{theorem}
\begin{proof}
We define the operator
\[
\mathcal{M}v(t,x):=\int_{\RR^d}Z(t,x-y)v_0(y)dy + \int_0^t \int_{\RR^d}Y(t-s,x-y)g(v(s,y))dyds
\]
on the Banach space $L_{\infty}((0,T)\times\RR^d)$, where $g$ is a non-decreasing Lipschitz function with $g(0) = 0$ and $v_0\in L_\infty(\RR^d)$. As in the proof of \cite[Lemma 1.3]{AE87}, we derive that the operator $\mathcal{M}$ has a unique fixed point $v$. Furthermore, $v\geq w$ whenever $v_0\geq w_0$, where $w$ is the fixed point associated with $w_0\in L_\infty(\RR^d)$. Our aim now is to apply this result to a sequence of functions $g_n$, such that for each $n\in\mathbb{N}$ they have the same properties as $g$ but with the additional constraint that their structure approximates the non-linear term $(\cdot)^\gamma$ on $[0,\infty)$. In accordance with our particular situation with $\gamma >1$, we need a sequence that allows us to control the derivative of the function $(\cdot)^\gamma$. For that purpose, we define
\[
g_n(r):=\begin{cases} 0 &\text{if }\quad r<0,\\ r^{\gamma} &\text{if }\quad 0\leq r\leq n,\\ a_n-b_ne^{-r} &\text{if }\quad r>n,\\ 
\end{cases}
\]
where $a_n,b_n$ are positive constants that guarantee the existence of $g_n'\geq 0$ on $\RR$ a.e. By construction we have that for all $n\in\mathbb{N}$ the constant Lipschitz of $g_n$ is $\gamma n^{\gamma-1}$, $g_n(0)=0$ and $g_n(r)=r^{\gamma}$ for $0\leq r\leq n$. Therefore, there exists a unique function $u_n\in L_{\infty}((0,T)\times\RR^d)$ such that $0\leq u_n$ and
\[
u_n(t,x)=\int_{\RR^d}Z(t,x-y)\left(u_0+\frac{1}{n}\right)(y)dy + \int_0^t \int_{\RR^d}Y(t-s,x-y)g_n(u_n(s,y))dyds,
\]
for $x\in\RR^d$ and $0<t<T$. Since $\frac{1}{n}\geq\frac{1}{n+1}$, we have that $u_{n+1}\leq u_{n}$. Thus, for almost every $(t,x)\in (0,T)\times\RR^d$, the sequence of real numbers $(u_n(t,x))_{n\in\mathbb{N}}$ is decreasing and bounded from below by zero. Consequently, we can define the function
\[
\widetilde{u}(t,x)=\lim_{n\rightarrow\infty}u_n(t,x)
\]
a.e. in $(0,T)\times\RR^d$. On the other hand, we have that
\[
\norm{u_n(t)}_\infty\leq\left\|u_0+\frac{1}{n}\right\|_\infty + \frac{\gamma n^{\gamma-1}}{\Gamma(\alpha)}\int_0^t (t-s)^{\alpha-1}\norm{u_n(s)}_{\infty}ds
\]
and Gronwall's inequality (see \cite[Corollary 2]{YGD07}) yields
\begin{align*}
\norm{u_n(t)}_\infty&\leq\left\|u_0+\frac{1}{n}\right\|_\infty E_{\alpha,1}\left(\gamma n^{\gamma-1}t^{\alpha}\right)\\
&\leq\left\|u_0+\frac{1}{n}\right\|_\infty E_{\alpha,1}\left(\gamma n^{\gamma-1}T^{\alpha}\right),\quad 0<t<T.
\end{align*}
Now, for small enough $0<T^*\leq T$ we can find $N\in\mathbb{N}$ such that 
\[
\left\|u_0+\frac{1}{N}\right\|_\infty E_{\alpha,1}\left(\gamma N^{\gamma-1}(T^*)^{\alpha}\right)\leq N.
\]
Therefore, for all $n\geq N$ it follows that $u_n(t,x)\leq N$, for $x\in\RR^d$ and $0<t<T^*$. This shows that
\[
u_n(t,x)=\int_{\RR^d}Z(t,x-y)\left(u_0+\frac{1}{n}\right)(y)dy + \int_0^t \int_{\RR^d}Y(t-s,x-y)u_n(s,y)^{\gamma}dyds,\quad n\geq N.
\]
We note that the non-linear integral term is dominated by $N^{\gamma}$ and the dominated convergence theorem implies that
\[
\widetilde{u}(t,x)=\int_{\RR^d}Z(t,x-y)u_0(y)dy + \int_0^t \int_{\RR^d}Y(t-s,x-y)\widetilde{u}(s,y)^{\gamma}dyds.
\]
Next, we show that $u=\widetilde{u}$ a.e. in $(0,T^*)$. Indeed,\\
\begin{align*}
|u(t,x)-\widetilde{u}(t,x)|&\leq \int_0^t \int_{\RR^d}Y(t-s,x-y)\left||u|^{\gamma-1}u(s,y)-\widetilde{u}(s,y)^{\gamma}\right|dyds\\
&=\int_0^t \int_{\RR^d}Y(t-s,x-y)\left||u|^{\gamma-1}u(s,y)-|\widetilde{u}|^{\gamma-1}\widetilde{u}(s,y)\right|dyds\\
&\lesssim\sup_{0\leq s<T^*}\left(\norm{u(s)}_\infty^{\gamma-1}+\norm{\widetilde{u}(s)}_\infty^{\gamma-1}\right) \int_0^t \int_{\RR^d}Y(t-s,x-y)|u(s,y)-\widetilde{u}(s,y)|dyds\\
&\leq C(T^*) \int_0^t \int_{\RR^d}Y(t-s,x-y)\norm{u(s)-\widetilde{u}(s)}_\infty dyds\\
&\leq C(T^*) \int_0^t\frac{(t-s)^{\alpha-1}}{\Gamma(\alpha)}\norm{u(s)-\widetilde{u}(s)}_\infty ds
\end{align*}
and thus
\[
\norm{u(t)-\widetilde{u}(t)}_\infty \leq \frac{C(T^*)}{\Gamma(\alpha)}\int_0^t(t-s)^{\alpha-1}\norm{u(s)-\widetilde{u}(s)}_\infty ds.
\]
By Gronwall's inequality we conclude the desired result.
\end{proof}

\section{Proof of the main result Theorem \ref{blowUp}}
\label{sec:4}
Firstly, we get the following estimates. Let $t>0$. Using the bounds given in Proposition \ref{cotasZ}, it is clear that 
\[
Z(t,x-y)\geq C t^{-\frac{\alpha d}{\beta}} e^{-\frac{\norm{x-y}^2}{4t}},\quad \Omega\leq 1.
\]
If $\Omega\geq 1$, we have that
\begin{align*}
Z(t,x-y)&\geq C t^{-\frac{\alpha d}{\beta}}\Omega^{-1-\frac{d}{\beta}}\\
&= C t^{-\frac{\alpha d}{\beta}}t^{\alpha+\frac{\alpha d}{\beta}}\norm{x-y}^{-\beta-d}\\
&= C t^{-\frac{\alpha d}{\beta}}t^{\alpha+\frac{\alpha d}{\beta}}(2\sqrt{t})^{-\beta-d}\left(\frac{\norm{x-y}}{2\sqrt{t}}\right)^{-\beta-d}\\
&\geq C t^{-\frac{\alpha d}{\beta}}t^{\alpha+\frac{\alpha d}{\beta}}(2\sqrt{t})^{-\beta-d}e^{-\frac{\norm{x-y}^2}{4t}},
\end{align*}
whenever $d\leq 3$. For larger dimensions, it is always possible to find a suitable constant $K>1$, depending on $\beta$ and $d$, such that $\left(\frac{\norm{x-y}}{2\sqrt{t}}\right)^{-\beta-d}\geq e^{-K\frac{\norm{x-y}^2}{4t}}$. From the hypothesis $\alpha=\frac{\beta}{2}$, it follows that
\[
Z(t,x-y)\geq C 2^{-\beta-d} t^{-\frac{\alpha d}{\beta}} e^{-\frac{\norm{x-y}^2}{4t}},\quad \Omega\geq 1,
\]
which means that
\begin{equation}
\label{Zvscampana}
Z(t,x-y)\geq C_1 t^{-\frac{\alpha d}{\beta}} e^{-\frac{\norm{x-y}^2}{4t}},
\end{equation}
for all $t>0$ and $x,y\in\RR^d$, with $C_1=\dfrac{C}{2^{\beta+d}}$.

We may assume without loss of generality that the constant $C$ of the Proposition \ref{cotasY} is the same as that of the Proposition \ref{cotasZ}. In this way, we have also derived
\begin{equation}
\label{Yvscampana}
Y(t-s,x-y)\geq C_1 (t-s)^{-\frac{\alpha d}{\beta}+\alpha-1} e^{-\frac{\norm{x-y}^2}{4(t-s)}},
\end{equation}
for all $0\leq s < t$ and $x,y\in\RR^d$. 

Now, we proceed by contradiction. We suppose that there exists a global non-trivial solution $u$ of \eqref{general}, according to Definition \ref{defSolution}. In this case, $u_0(y_0)>0$ for some $y_0\in\RR^d$. The continuity of $u_0$ implies that
\[
u_0(y)> C_0, \quad \forall y\in B(y_0,\delta),
\] 
with some $\delta >0$ and $C_0=\dfrac{u_0(y_0)}{2}$.

The representation \eqref{integral} for $u$ is
\[
u(t,x) = \int_{\RR^d}Z(t,x-y)u_0(y)dy + \int_0^t \int_{\RR^d}Y(t-s,x-y)u(s,y)^{\gamma}dyds
\]
for all $x\in\RR^d$ and $0<t< T$. We note that, given the assumption made, $T$ can be arbitrarily large. As in Section \ref{sec:3}, we define
\[
u_1(t,x):= \int_{\RR^d}Z(t,x-y)u_0(y)dy 
\]
and
\[
u_2(t,x):=\int_0^t \int_{\RR^d}Y(t-s,x-y)u(s,y)^{\gamma}dyds.
\]
Using \eqref{Zvscampana}, it follows that
\begin{align*}
u_1(t,x)&\geq C_1 t^{-\frac{\alpha d}{\beta}} \int_{\RR^d}e^{-\frac{\norm{x-y}^2}{4t}} u_0(y)dy\\
&\geq C_1C_0 t^{-\frac{\alpha d}{\beta}}\int_{B(y_0,\delta)}e^{-\frac{\norm{x-y}^2}{4t}}dy\\
&\geq C_1C_0 t^{-\frac{\alpha d}{\beta}}e^{-\frac{\norm{x-y_0}^2}{2t}}\int_{B(y_0,\delta)}e^{-\frac{\norm{y-y_0}^2}{2t}}dy
\end{align*}
and we obtain
\begin{equation}
\label{EstiPartLineal}
u_1(t,x)\geq C_2 t^{-\frac{\alpha d}{\beta}}e^{-\frac{\norm{x}^2}{t}}, \quad t>1,\quad x\in\RR^d.
\end{equation}
Let $H$ be the heat kernel 
\[
H(t,x)=\frac{1}{(4\pi t)^{\frac{d}{2}}}e^{-\frac{\norm{x}^2}{4t}},\quad t>0,\quad x\in\RR^d.
\]
Using the fact that
\[
\int_{\RR^d}H(t,x)dx=1,
\]
we define the function
\begin{equation}
\label{F}
F(t)=\int_{\RR^d}H(t,x)u(t,x)dx,\quad t>0,
\end{equation}
and splitting the integral into two parts we see that
\[
F(t)=\int_{\RR^d}H(t,x)u_1(t,x)dx+\int_{\RR^d}H(t,x)u_2(t,x)dx.
\]
In the first integral we use the estimate \eqref{EstiPartLineal}, for obtaining
\[
F(t)\geq C_3 t^{-\frac{\alpha d}{\beta}}+\int_{\RR^d}H(t,x)u_2(t,x)dx
\]
whenever $t>1$. 

In the second integral, we use the fact that (see \cite[Theorem 2.14]{SV22})
\[
\frac{1}{g_\alpha(t)}\displaystyle\int_{\RR^d}Y(t,x)dx=1,\quad t>0.
\]
Jensen's inequality and Fubini's theorem yield 
\begin{align*}
&\int_{\RR^d}H(t,x)u_2(t,x)dx\\
&=\int_{\RR^d}H(t,x)\left[\int_0^t \int_{\RR^d}Y(t-s,x-y)u(s,y)^{\gamma}dy ds\right] dx\\
&=\int_0^t g_\alpha(t-s)\int_{\RR^d}H(t,x)\left[\int_{\RR^d}\frac{1}{g_\alpha(t-s)} Y(t-s,x-y)u(s,y)^{\gamma}dy\right] dx ds\\
&\geq \int_0^t g_\alpha(t-s)\int_{\RR^d}H(t,x)\left[\int_{\RR^d} \frac{1}{g_\alpha(t-s)} Y(t-s,x-y)u(s,y)dy\right]^{\gamma} dx ds\\
&= \int_0^t \left(g_\alpha(t-s)\right)^{1-\gamma}\int_{\RR^d}H(t,x)\left[\int_{\RR^d}Y(t-s,x-y)u(s,y)dy\right]^{\gamma} dx ds\\
&\geq \int_0^t \left(g_\alpha(t-s)\right)^{1-\gamma}\left[\int_{\RR^d}H(t,x)\int_{\RR^d}Y(t-s,x-y)u(s,y)dy\; dx\right]^{\gamma} ds\\
&\geq \int_0^t \left(g_\alpha(t-s)\right)^{1-\gamma}\left\{\int_{\RR^d}\left[\int_{\RR^d}H(t,x)Y(t-s,x-y)dx \right] u(s,y)dy\right\}^{\gamma} ds.
\end{align*}
The expression in the square brackets can be estimated with \eqref{Yvscampana}, i.e.,
\begin{align*}
&\int_{\RR^d}H(t,x)Y(t-s,x-y)dx\\
&\geq C_1  (t-s)^{-\frac{\alpha d}{\beta}+\alpha-1} \int_{\RR^d}H(t,x)e^{-\frac{\norm{x-y}^2}{4(t-s)}}dx\\
&= C_1 (4\pi s)^{-\frac{d}{2}}e^{-\frac{\norm{y}^2}{4s}}\left(\frac{s}{t}\right)^{\frac{d}{2}} (t-s)^{-\frac{\alpha d}{\beta}+\alpha-1} \int_{\RR^d}e^{\frac{\norm{y}^2}{4s}-\frac{\norm{x}^2}{4t}-\frac{\norm{x-y}^2}{4(t-s)}}dx. 
\end{align*}
Proceeding in the same way as in \cite[page 42]{Bei11}, with $\alpha=\frac{\beta}{2}$, we get
\[
\int_{\RR^d}H(t,x)Y(t-s,x-y)dx \geq C_4 (4\pi s)^{-\frac{d}{2}}e^{-\frac{\norm{y}^2}{4s}}\left(\frac{s}{t}\right)^{\frac{d}{2}} (t-s)^{\alpha-1}
\]
and thus
\begin{align*}
&\left\{\int_{\RR^d}\left[\int_{\RR^d}H(t,x)Y(t-s,x-y)dx \right] u(s,y)dy\right\}^{\gamma}\\ 
&\geq C_4^\gamma (t-s)^{(\alpha-1)\gamma} \left(\frac{s}{t}\right)^{\frac{d}{2}\gamma}\left\{\int_{\RR^d}(4\pi s)^{-\frac{d}{2}}e^{-\frac{\norm{y}^2}{4s}}u(s,y)dy\right\}^\gamma\\ 
&=C_4^\gamma(t-s)^{(\alpha-1)\gamma} \left(\frac{s}{t}\right)^{\frac{d}{2}\gamma}F^\gamma(s)
\end{align*}
for $0<s<t$.
It follows that
\[
\int_{\RR^d}H(t,x)u_2(t,x)dx\geq C_4^\gamma\int_0^t \left(g_\alpha(t-s)\right)^{1-\gamma}(t-s)^{(\alpha-1)\gamma} \left(\frac{s}{t}\right)^{\frac{d}{2}\gamma}F^\gamma(s) ds
\]
and hence
\begin{equation*}
%\label{Facotada}
F(t) \geq \frac{C_3}{t^{\frac{d}{2}}}+C_5 \frac{t^{\alpha-1}}{t^{\frac{d}{2}\gamma}}\int_0^t s^{\frac{d}{2}\gamma}F^\gamma(s) ds
\end{equation*}
for all $t>1$. Consequently,
\begin{equation}
\label{FacotadaCF}
t^{\frac{d}{2}\gamma}t^{1-\alpha} F(t) \geq C_3 t^{\frac{d}{2}(\gamma-1)}t^{1-\alpha}+C_5 \int_0^t s^{\frac{d}{2}\gamma}F^\gamma(s) ds.
\end{equation}
Defining the r.h.s. of this expression as $f(t)$, $t>1$, we have that
\begin{equation}
\label{cotaf}
f(t)\geq C_3 t^{\frac{d}{2}(\gamma-1)}t^{1-\alpha}
\end{equation}
and that
\begin{equation}
\label{cotaf'}
f'(t)\geq C_5 t^{\frac{d}{2}\gamma}F^\gamma(t).
\end{equation}
From \eqref{FacotadaCF} it follows that
\begin{align*}
f'(t)&\geq C_5 t^{\frac{d}{2}\gamma}\left(\frac{f(t)}{t^{\frac{d}{2}\gamma+1-\alpha}}\right)^\gamma\\
&=C_5 t^{\frac{d}{2}\gamma(1-\gamma)-(1-\alpha)\gamma}f^\gamma(t).
\end{align*}
Therefore,
\[
f'(t)f^{-\gamma}(t)\geq C_5 t^{\frac{d}{2}\gamma(1-\gamma)-(1-\alpha)\gamma}
\]
and
\[
\int_t^T f'(s)f^{-\gamma}(s)ds \geq C_5\int_t^T s^{\frac{d}{2}\gamma(1-\gamma)-(1-\alpha)\gamma}ds
\]
with $T>t$. From here, we get that
\[
\frac{f^{1-\gamma}(t)}{\gamma-1}\geq C_5\int_t^T s^{\frac{d}{2}\gamma(1-\gamma)-(1-\alpha)\gamma}ds
\]
and using \eqref{cotaf} we also obtain the estimate
\[
\frac{f^{1-\gamma}(t)}{\gamma-1}\leq\frac{C_3^{1-\gamma}}{\gamma-1}t^{-\frac{d}{2}(1-\gamma)^2-(1-\alpha)(\gamma-1)}.
\]
This implies that
\begin{equation}
 \label{exponentes}
\frac{C_3^{1-\gamma}}{\gamma-1}t^{-\frac{d}{2}(1-\gamma)^2-(1-\alpha)(\gamma-1)}\geq C_5\int_t^T s^{-\frac{d}{2}\gamma(\gamma-1)-(1-\alpha)\gamma}ds.
 \end{equation}
Next we analyse the r.h.s. of \eqref{exponentes}, according to the following cases with $a:=d-2(1-\alpha)$.\\

For the case $1<\gamma\leq \frac{a}{d}+\frac{2}{d\gamma}$, we have 
\begin{align*}
\gamma\leq \frac{a}{d}+\frac{2}{d\gamma}&\Rightarrow d\gamma^2\leq a\gamma+2\\
&\Leftrightarrow d\gamma^2+2(1-\alpha)\gamma-d\gamma-2\leq 0\\
&\Leftrightarrow -\frac{d\gamma}{2}(\gamma-1)-(1-\alpha)\gamma+1\geq 0,
\end{align*}
which yields a contradiction for large enough $T$.

For the case $\frac{a}{d}+\frac{2}{d\gamma}<\gamma< \frac{a}{d}+\frac{2}{d}$, we write the expression \eqref{exponentes} as 
\[
\frac{C_3^{1-\gamma}}{\gamma-1}t^{-\frac{d}{2}(1-\gamma)^2-(1-\alpha)(\gamma-1)}\geq C_5\frac{t^{-\frac{d}{2}\gamma(\gamma-1)-(1-\alpha)\gamma+1}-T^{-\frac{d}{2}\gamma(\gamma-1)-(1-\alpha)\gamma+1}}{\frac{d}{2}\gamma(\gamma-1)+(1-\alpha)\gamma-1}.
\]
Besides
\begin{align*}
\gamma<\frac{a}{d}+\frac{2}{d}&\Rightarrow d\gamma<d-2(1-\alpha)+2\\
&\Leftrightarrow -1<-\frac{d}{2}(\gamma-1)-(1-\alpha)\\
&\Leftrightarrow \frac{d\gamma}{2}(\gamma-1)+(1-\alpha)\gamma-1<\frac{d}{2}(\gamma-1)^2+(1-\alpha)(\gamma-1),
\end{align*}
which is a contradiction for large enough $t$ and $T\rightarrow \infty$.

For the critical case $\gamma=1+\frac{\beta}{d}$, we use the facts that 
\[
u(t,x)^\gamma\geq u_1(t,x)^\gamma
\]
and
\[
u(t,x)\geq u_2(t,x),
\]
together with the estimates \eqref{Yvscampana} and \eqref{EstiPartLineal}. Therefore, for $t>2$, we get
\begin{align*}
u(t,x)&\geq \int_1^{\frac{t}{2}} \int_{\RR^d}Y(t-s,x-y)u(s,y)^{\gamma}dyds\\
&\geq C_1 C_2^\gamma\int_1^{\frac{t}{2}}(t-s)^{-\frac{\alpha d}{\beta}+\alpha-1}\int_{\RR^d} e^{-\frac{\norm{x-y}^2}{4(t-s)}} s^{-\frac{\alpha d \gamma}{\beta}}e^{-\frac{\gamma\norm{y}^2}{s}}dyds\\
&=\frac{C_1 C_2^\gamma}{t^{\frac{d}{2}}}e^{-\frac{\norm{x}^2}{t}}\int_1^{\frac{t}{2}}\frac{(t-s)^{\alpha-1}t^{\frac{d}{2}}}{(t-s)^{\frac{d}{2}}s^{\frac{d+\beta}{2}}}\int_{\RR^d} e^{\frac{\norm{x}^2}{t}-\frac{\norm{x-y}^2}{4(t-s)}-\frac{\gamma\norm{y}^2}{s}}dyds\\
&\geq\frac{C_1 C_2^\gamma}{t^{\frac{d}{2}}}e^{-\frac{\norm{x}^2}{t}}\int_1^{\frac{t}{2}}\frac{(t-s)^{\alpha-1}t^{\frac{d}{2}}}{(t-s)^{\frac{d}{2}}s^{\frac{d}{2}+\alpha}}\int_{\RR^d} e^{\frac{\norm{x}^2}{t}-\frac{\norm{x}^2}{t}-\frac{\norm{y}^2}{2(t-s)}-\frac{\gamma\norm{y}^2}{s}}dyds\\
&\geq\frac{C_6}{t^{\frac{d}{2}}}e^{-\frac{\norm{x}^2}{t}}\int_1^{\frac{t}{2}}\frac{(t-s)^{\alpha-1}}{s^{\alpha}}ds\\
&\geq \frac{C_6}{t^{\frac{d}{2}}}e^{-\frac{\norm{x}^2}{t}}\int_1^{\frac{t}{2}}\frac{1}{t-s}ds
\end{align*}  
and hence
\[
u(t,x)\geq \frac{C_6}{t^{\frac{d}{2}}}e^{-\frac{\norm{x}^2}{t}}\ln\left(2-\frac{2}{t}\right).
\]
%for large enough $t$.
Using this and \eqref{F}, we obtain that
\begin{equation}
\label{nuevacota}
F(t)\geq \frac{C_7}{t^{\frac{d}{2}}}\ln\left(2-\frac{2}{t}\right).
\end{equation}
Now,
\begin{align*}
t^{\frac{d}{2}\gamma}t^{1-\alpha}  F(t)&=\frac{1}{2}t^{\frac{d}{2}\gamma}t^{1-\alpha} F(t)+\frac{1}{2}t^{\frac{d}{2}\gamma}t^{1-\alpha} F(t)\\
&\geq \frac{C_7}{2}\frac{t^{\frac{d}{2}\gamma}t^{1-\alpha}}{t^{\frac{d}{2}}}\ln\left(2-\frac{2}{t}\right)+\frac{C_5}{2}\int_0^t s^{\frac{d}{2}\gamma}F^\gamma(s) ds,
\end{align*}
where \eqref{nuevacota} yields the bound for the first term and the second term comes from \eqref{FacotadaCF}. The critical value of $\gamma$ yields
\[
t^{\frac{d}{2}\gamma}t^{1-\alpha}  F(t)\geq \frac{C_7}{2}t\ln\left(2-\frac{2}{t}\right)+\frac{C_5}{2}\int_0^t s^{\frac{d}{2}\gamma}F^\gamma(s) ds.
\]
Defining the r.h.s. of this expression as the new $f(t)$, $t>1$, we proceed as before but using
\[
f(t)\geq C_8 t\ln\left(2-\frac{2}{t}\right)
\]
and
\[
f'(t)\geq C_9 t^{\frac{d}{2}\gamma}F^\gamma(t)
\]
instead of \eqref{cotaf} and \eqref{cotaf'}, respectively, with $C_8=\frac{C_7}{2}$ and $C_9=\frac{C_5}{2}$. The resulting expression, instead of \eqref{exponentes}, is
\[
\frac{C_8^{1-\gamma}}{\gamma-1}t^{1-\gamma}\ln^{1-\gamma}\left(2-\frac{2}{t}\right)\geq C_9\int_t^T s^{-\frac{d}{2}\gamma(\gamma-1)-(1-\alpha)\gamma}ds
\]
or, in this case,
\[
\frac{C_8^{1-\gamma}}{\gamma-1}t^{1-\gamma}\ln^{1-\gamma}\left(2-\frac{2}{t}\right)\geq C_9\int_t^T s^{-\gamma}ds.
\]
This implies, as $T\rightarrow\infty$, that
\[
C_8^{1-\gamma}\ln^{1-\gamma}\left(2-\frac{2}{t}\right)\geq C_9,
\]
which is a contradiction whenever the initial condition is sufficiently large at the point $y_0$.

So far we note that in this proof we do not require that $u$ satisfies \eqref{general}. Hence, any positive mild solution $u$ can only be local under the assumptions of Theorem \ref{blowUp}. In this context, let 
\[
\widetilde{T}=\sup\left\{T>0: u\in C([0,T];L_p(\RR^d))\cap L_{\infty}((0,T)\times\RR^d)\text{ is a positive mild solution of }\eqref{general}\right\}. 
\]
Previous work implies that $\widetilde{T}<+\infty$. Suppose that $\lim_{t\rightarrow \widetilde{T}^-}\norm{u(t)}_\infty<+\infty$. Since $u_0\in L_\infty(\RR^d)$, it follows that there exists $M>0$ such that $\norm{u(t)}_\infty\leq M$ for all $t\in[0,\widetilde{T})$. We choose a sequence $t_n\rightarrow \widetilde{T}$ as $n\rightarrow \infty$, with $t_n<\widetilde{T}$ for all $n\in\mathbb{N}$. We suppose $\frac{1}{2}\widetilde{T}<t_m<t_n$ without loss of generality, with $n,m\geq N$ for some $N\in\mathbb{N}$. As in the proof of \cite[Theorem 3.1]{SV22}, we find that
\begin{align*}
\norm{u(t_n)-u(t_m)}_p\lesssim &(t_n-t_m)t_m^{-1}\norm{u_0}_p\\
&+M^{\gamma-1}\int_0^{t_m}\norm{Y(t_n-s)-Y(t_m-s)}_1\norm{u(s)}_p ds\\
&+M^{\gamma-1}\int_{t_m}^{t_n}\norm{Y(t_n-s)}_1\norm{u(s)}_p ds.
\end{align*}
On the other hand, for any $t\in[0,\widetilde{T})$ we see that
\[
\norm{u(t)}_p\leq\norm{u_0}_p+\frac{M^{\gamma-1}}{\Gamma(\alpha)}\int_0^{t}(t-s)^{\alpha-1}\norm{u(s)}_p ds
\]
and Gronwall's inequality (\cite[Corollary 2]{YGD07}) yields
\[
\norm{u(t)}_p\leq\norm{u_0}_p E_{\alpha,1}(M^{\gamma-1}t^\alpha),\quad 0\leq t<\widetilde{T}.
\]
This shows that $\norm{u(t)}_p\leq \norm{u_0}_p E_{\alpha,1}(M^{\gamma-1}\widetilde{T}^\alpha)=:K$ for all $t\in[0,\widetilde{T})$. Thus,
\begin{align*}
\norm{u(t_n)-u(t_m)}_p\lesssim &(t_n-t_m)t_m^{-1}\norm{u_0}_p\\
&+M^{\gamma-1}K\int_0^{t_m}\norm{Y(t_n-s)-Y(t_m-s)}_1 ds\\
&+M^{\gamma-1}K\int_{t_m}^{t_n}\norm{Y(t_n-s)}_1 ds.
\end{align*}
The integral over $[0,t_m]$ can be estimated, using \cite[Theorems 2.10 and 2.14]{SV22}, as follows:
\begin{align*}
&\int_0^{t_m}\norm{Y(t_n-s)-Y(t_m-s)}_1 ds=\int_0^{t_m}\norm{Y(t_n-t_m+s)-Y(s)}_1 ds\\
&\leq \int_0^{\infty}\norm{Y(t_n-t_m+s)-Y(s)}_1 ds\\
&= \int_0^{t_n-t_m}\norm{Y(t_n-t_m+s)-Y(s)}_1 ds+\int_{t_n-t_m}^{\infty}\norm{Y(t_n-t_m+s)-Y(s)}_1 ds\\
&\leq \int_0^{t_n-t_m}\norm{Y(t_n-t_m+s)}_1ds+\int_0^{t_n-t_m}\norm{Y(s)}_1 ds+\int_{t_n-t_m}^{\infty}\norm{Y(t_n-t_m+s)-Y(s)}_1 ds\\
&\lesssim \int_0^{t_n-t_m}(t_n-t_m+s)^{\alpha-1}ds+\int_0^{t_n-t_m}s^{\alpha-1}ds+\int_{t_n-t_m}^{\infty}(t_n-t_m)s^{\alpha-2} ds\\
&\lesssim(t_n-t_m)^\alpha.
\end{align*}
Consequently,
\[
\norm{u(t_n)-u(t_m)}_p\lesssim (t_n-t_m)\widetilde{T}^{-1}\norm{u_0}_p+M^{\gamma-1}K(t_n-t_m)^\alpha
\]
and thus $(u(t_n))_{n\in\mathbb{N}}$ represents a Cauchy sequence in $L_p(\RR^d)$. We define $u(\widetilde{T}):=\lim_{t\rightarrow\widetilde{T}^-}u(t)$. From \cite[Theorem 3.12]{Rud87} it follows that $\norm{u(\widetilde{T})}_\infty\leq M$ and that $u(\widetilde{T})\geq 0$. Next, as in the proof of \cite[Theorem 3.2]{ZS15}, we define the operator
\[
\mathcal{M}v(t):=Z(t)\star u_0+\int_0^{\widetilde{T}}Y(t-s)\star u^\gamma(s)ds+\int_{\widetilde{T}}^t Y(t-s)\star |v(s)|^{\gamma-1}v(s)ds
\]
on the Banach space 
\[
E_\tau=C([\widetilde{T},\widetilde{T}+\tau];L_p(\RR^d))\cap L_\infty([\widetilde{T},\widetilde{T}+\tau)\times\RR^d),
\]
with some $\tau>0$ and the norm
\[
\norm{v}_{E_\tau}=\sup_{t\in [\widetilde{T},\widetilde{T}+\tau]}\norm{v(t)}_p+\sup_{(t,x)\in[\widetilde{T},\widetilde{T}+\tau)\times\RR^d}\abs{v(t,x)}.
\]
It is straightforward to see that $\mathcal{M}:E_\tau\rightarrow E_\tau$ is well defined and that $\mathcal{M}v(\widetilde{T})=u(\widetilde{T})$. Besides, for $v,w\in E_\tau$ we have that 
\begin{align*}
|\mathcal{M}v(t,x)-\mathcal{M}w(t,x)|&\leq\norm{\mathcal{M}v(t)-\mathcal{M}w(t)}_\infty\\
&\leq\int_{\widetilde{T}}^t\norm{Y(t-s)}_1\norm{|v(s)|^{\gamma-1}v(s)-|w(s)|^{\gamma-1}w(s)}_\infty ds\\
&\lesssim (\norm{v}_{E_\tau}+\norm{w}_{E_\tau})^{\gamma-1}\int_{\widetilde{T}}^t(t-s)^{\alpha-1}\norm{v(s)-w(s)}_\infty ds\\
&\lesssim (\norm{v}_{E_\tau}+\norm{w}_{E_\tau})^{\gamma-1}\norm{v-w}_{E_\tau}(t-\widetilde{T})^\alpha,
\end{align*}
and hence
\[
\norm{\mathcal{M}v(t)-\mathcal{M}w(t)}_\infty\lesssim (\norm{v}_{E_\tau}+\norm{w}_{E_\tau})^{\gamma-1}\norm{v-w}_{E_\tau}\tau^\alpha,\quad t\in [\widetilde{T},\widetilde{T}+\tau).
\]
Similarly,
\[
\norm{\mathcal{M}v(t)-\mathcal{M}w(t)}_p\lesssim (\norm{v}_{E_\tau}+\norm{w}_{E_\tau})^{\gamma-1}\norm{v-w}_{E_\tau}\tau^\alpha,\quad t\in [\widetilde{T},\widetilde{T}+\tau].
\]
Therefore, there exists $C_{10}>0$ such that
\begin{equation}
\label{EstContraccion}
\norm{\mathcal{M}v-\mathcal{M}w}_{E_\tau}\leq C_{10}\tau^\alpha(\norm{v}_{E_\tau}+\norm{w}_{E_\tau})^{\gamma-1}\norm{v-w}_{E_\tau},\quad v,w\in E_\tau.
\end{equation}
We also find that
\begin{align*}
\left\|Z(t)\star u_0+\int_0^{\widetilde{T}}Y(t-s)\star u^\gamma(s)ds\right\|_\infty&\leq\norm{Z(t)\star u_0}_\infty+\int_0^{\widetilde{T}}\norm{Y(t-s)}_1\norm{u^\gamma(s)}_\infty ds\\
&\lesssim \norm{u_0}_\infty+M^{\gamma}\int_0^{\widetilde{T}}(t-s)^{\alpha-1}ds\\
&\lesssim \norm{u_0}_\infty+M^{\gamma}(t^{\alpha}-(t-\widetilde{T})^\alpha)\\
&\lesssim \norm{u_0}_\infty+M^{\gamma}\widetilde{T}^\alpha
\end{align*}
and that
\[ 
\left\|Z(t)\star u_0+\int_0^{\widetilde{T}}Y(t-s)\star u^\gamma(s)ds\right\|_p \lesssim \norm{u_0}_p+M^{\gamma-1}K\widetilde{T}^\alpha,
\]
that is, there exists $C_{11}>0$ such that
\begin{equation}
\label{EstFuncion}
\left\|Z(t)\star u_0+\int_0^{\widetilde{T}}Y(t-s)\star u^\gamma(s)ds\right\|_{E_\tau} \leq C_{11}\left(\norm{u_0}_\infty+\norm{u_0}_p+M^{\gamma-1}(M+K)\widetilde{T}^\alpha\right).
\end{equation}
Let $R=2C_{11}\left(\norm{u_0}_\infty+\norm{u_0}_p+M^{\gamma-1}(M+K)\widetilde{T}^\alpha\right)$. If we consider the closed ball
\[
B_{E_\tau}:=\{w\in E_\tau:\norm{w}_{E_\tau}\leq R \},
\]
then estimates \eqref{EstContraccion}, with $v=0$, and \eqref{EstFuncion} show that $\mathcal{M}: B_{E_\tau}\rightarrow B_{E_\tau}$ is a contraction whenever $\tau$ is small enough (see \cite[Theorem 3.1]{SV22}), thus showing that $\mathcal{M}$ has a unique fixed point $w'\in B_{E_\tau}$. Moreover, since $u\geq 0$ we obtain that $w'\geq 0$ in $[\widetilde{T},\widetilde{T}+\tau)\times\RR^d$ following the same arguments as in the proof of Theorem \ref{positividad}, but one must now use the fact that 
\[
v_n(t)=Z(t)\star \left(u_0+\frac{1}{n}\right)+\int_0^{\widetilde{T}}Y(t-s)\star \left(u+\frac{1}{n}\right)^\gamma(s)ds+\int_{\widetilde{T}}^t Y(t-s)\star g_n(v_n(s))ds,
\]
for all $n\in\mathbb{N}$ and $v_n\in L_\infty (\widetilde{T},\widetilde{T}+\tau)\times\RR^d$. However, this leads to a contradiction with the definition of $\widetilde{T}$, and therefore $\lim_{t\rightarrow \widetilde{T}^-}\norm{u(t)}_\infty =+\infty$.

\begin{flushright}
$\square$
\end{flushright}

The final result of this section deals with the case $\gamma > 1+\frac{\beta}{d}$. For this purpose, as in Section \ref{sec:3}, we set $$\kappa=\begin{cases}\frac{d}{\beta}, & d>\beta,\\ 1, & otherwise \end{cases}.$$ 
We also define $H_2^\beta(\RR^d):=\overline{C_0^\infty(\RR^d)}^{\norm{\cdot}_{\Psi_{\beta},L_2}}$, with the closure being respect to the graph norm $\norm{\cdot}_{\Psi_{\beta},L_2}^2=\norm{\cdot}_{2}^2 + \norm{\Psi_{\beta}(-i\nabla)(\cdot)}_{2}^2$.

\begin{theorem}
\label{SolucionGlobal}
Let $\alpha\in (0,1)$ and $\beta\in (0,2)$. Assume the hypothesis $(\mathcal{H}_1)$ holds. Suppose that $\gamma >1+ \frac{\beta}{d}$, that $\max\left(1,\kappa,\frac{d(\gamma-1)}{\beta}\right)<p<\infty$ and that $1=p'<\frac{d}{\beta}(\gamma-1)$ whenever $d<\beta$, or $\frac{d}{\beta}<p'<\frac{d}{\beta}(\gamma-1)$ whenever $d\geq\beta$. If $u_0\in L_1(\RR^d) \cap H_2^\beta(\RR^d) \cap L_\infty(\RR^d)$ is sufficiently small and non-negative, then there exists a global solution $u$ to \eqref{general} in the sense of Definition \ref{defSolution} and the optimal time decay estimate
\[
\norm{u(t)}_1+t^{\frac{\alpha d}{\beta}\left(\frac{1}{p'}-\frac{1}{p}\right)}\norm{u(t)}_p+t^{\frac{\alpha d}{\beta p'}}\norm{u(t)}_\infty\lesssim\left(\norm{u_0}_1+\norm{u_0}_p +\norm{u_0}_\infty\right) 
\]  
is true for all $t\geq 1$.
\end{theorem}

\begin{remark}
Whenever $d\leq\beta$, the existence of parameter $p'$ follows from the fact that $\gamma >1+ \frac{\beta}{d}$. However, in the case $d>\beta$ one can not generally guarantee the existence of $p'$. 
\end{remark}

\begin{proof}
We consider the Banach space
\[
E:=C([0,\infty); L_p(\RR^d) \cap L_1(\RR^d)) \cap L_\infty((0,\infty); L_{\infty}(\RR^d)),
\]
with the norm
	\[
	\left\|v\right\|_{E}:=\sup_{t\geq 0}\left(\seq{t}^{\frac{\alpha d}{\beta}\left(\frac{1}{p'}-\frac{1}{p}\right)}\left\|v(t,\cdot)\right\|_p + \left\|v(t,\cdot)\right\|_1\right) + \sup_{t>0}\{t\}^{\frac{\alpha d}{\beta p}}\seq{t}^{\frac{\alpha d}{\beta p'}}\left\|v(t,\cdot)\right\|_{\infty},
	\]
where $\seq{t}:=\sqrt{1+t^2}$ and $\{t\}:=\dfrac{t}{\sqrt{1+t^2}}$. We define on $E$ the operator
\[
\mathcal{M}(v)(t,x):=\int_{\RR^d} Z(t,x-y)u_0(y)dy + \int_0^t\int_{\RR^d} Y(t-s,x-y)|v(s,y)|^{\gamma-1}v(s,y)dyds
\]
and similar arguments as in \cite[Sections 3 and 4]{SV22} show that
\[
\mathcal{M}(v)\in C([0,\infty); L_p(\RR^d) \cap L_1(\RR^d))
\]
and that
$$\left\|Z(t,\cdot)\star u_0\right\|_{\infty}\leq\left\|Z(t,\cdot)\right\|_1\left\|u_0\right\|_\infty=\norm{u_0}_\infty,\quad t>0.$$
For $0<t\leq 1$ we have that
\begin{align*}
\left\|\displaystyle\int_0^t Y(t-s,\cdot)\star|v(s,\cdot)|^{\gamma-1}v(s,\cdot)ds\right\|_{\infty}
&\leq \displaystyle\int_0^{t} \norm{Y(t-s,\cdot)}_{1}\norm{|v(s,\cdot)|^{\gamma-1}v(s,\cdot)}_\infty ds\\
&\lesssim\sup_{(t,x)\in [0,1]\times\RR^d}\abs{v(t,x)}^{\gamma}\int_0^t (t-s)^{\alpha-1}ds\\
&\lesssim\sup_{(t,x)\in [0,1]\times\RR^d}\abs{v(t,x)}^{\gamma}
\end{align*}
and for $t>1$ we obtain (see \cite[Section 4]{SV22})
\begin{align*}
&\left\|\displaystyle\int_0^t Y(t-s,\cdot)\star|v(s,\cdot)|^{\gamma-1}v(s,\cdot)ds\right\|_\infty\\
&\leq\displaystyle\int_0^{t} \norm{Y(t-s,\cdot)}_{\frac{p'}{p'-1}}\norm{|v(s,\cdot)|^{\gamma-1}v(s,\cdot)}_{p'} ds\\
&\lesssim \norm{v}_E^\gamma\displaystyle\int_0^{t}(t-s)^{-\frac{\alpha d}{\beta p'}+\alpha-1}s^{-\frac{\alpha d}{\beta p}(\gamma-1)}\langle s\rangle^{-\frac{\alpha d}{\beta}(\gamma-1)\left(\frac{1}{p'}-\frac{1}{p}\right)}ds\\
&\lesssim \norm{v}_E^\gamma t^{-\frac{\alpha d}{\beta p'}}\\
&\leq \norm{v}_E^\gamma.
\end{align*}
This proves that
\[
\mathcal{M}(v)\in L_\infty((0,\infty); L_\infty(\RR^d)).
\]
Besides, as in \cite[Section 4]{SV22} one finds that
$$\norm{Z\star u_0}_{E}\leq C_1\left(\norm{u_0}_1+\norm{u_0}_p +\norm{u_0}_\infty\right)$$
and that the operator $\mathcal{M}$ is a contraction in the closed ball $B_R=\{v\in E:\norm{v}_E\leq R\}$ of radius $R=2C_1\left(\norm{u_0}_1+\norm{u_0}_p +\norm{u_0}_\infty\right)$. Consequently there exists a fixed point $\widetilde{u}$ which is unique in $E$ because of Gronwall's inequality (\cite[Corollary 2]{YGD07}). 

Let $T>0$. We define the Volterra equation
$$u(t)=u_0+g_{\alpha}*|\widetilde{u}|^{\gamma-1}\widetilde{u}(t)+g_{\alpha}*Au(t), \quad 0\leq t\leq T,$$
and by proceeding as in \cite[Section 5]{SV22}, since $u_0\in H_2^\beta(\RR^d)$, we find that there exists a unique strong solution $u\in L_2([0,T];H_2^\beta(\RR^d))$, and it satisfies the variation of parameters formula
\[
u(t)=\frac{d}{dt}\int_0^t S(s)\left(u_0+g_{\alpha}*|\widetilde{u}|^{\gamma-1}\widetilde{u}\right)(t-s) ds.
\]   
On the other hand, similar arguments as in \cite[Lemma 5.1]{SV22} show that the fixed point $\widetilde{u}$ satisfies
\[
\widetilde{u}(t)=\frac{d}{dt}\int_0^t S(s)\left(u_0+g_{\alpha}*|\widetilde{u}|^{\gamma-1}\widetilde{u}\right)(t-s) ds
\]   
and therefore $\widetilde{u}=u$. This holds for any $T>0$ which implies that $u$ is global.
\end{proof}

\begin{remark}
Since $u_0\in L_\infty(\RR^d)$, Theorem \ref{positividad} guarantees the positivity of the global solution $u$ on $[0,T)$ for some $T>0$. 
\end{remark}

\section*{Acknowledgements}
The authors are very grateful to the referees for their valuable comments and suggestions, which helped to improve the quality of the paper significantly.

\section*{Funding}
The authors were partially supported by Chilean research grant Fondo Nacional de Desarrollo Científico y Tecnológico, FONDECYT 1190255.

\section*{Conflict of interest}
The authors declare that they have no conflict of interest.

\section*{Availability of data and material}
Not applicable. No datasets were generated or analysed during the current study.

\section*{Code availability}
Not applicable.

\end{document}